\documentclass{amsart} 
\usepackage{amsthm} 
\usepackage{amsmath} 
\usepackage{amsfonts} 
\usepackage{ifpdf}
\ifpdf
   \usepackage{url}
   \usepackage[pdftex]{graphicx}
    \usepackage[pdftex,
bookmarks=true,
bookmarksnumbered=true,
hypertexnames=false,
breaklinks=true,
linkbordercolor={0 0 1},
pdfborder={0 0 1}]{hyperref}
 \else
   \usepackage{graphicx}
 \fi
\renewcommand{\c}[1]{\lceil #1 \rceil} 
\newcommand{\f}[1]{\lfloor #1 \rfloor}         
\newcommand{\ch}{\mbox {\bf 1}}

\newcommand{\R}{{\mathbb{R}}} %
  
\newcommand{\Z}{{\mathbb{Z}}}  
\newcommand{\N}{{\mathbb{N}}}  
  
\newcommand{\ol}[1]{\overline{#1}}

\newcommand{\cal}{\mathcal} 
\newcommand{\sign}{\text{sgn}} 
\newtheorem{itdefinition}{Definition}  
  
\newtheorem*{itdefinition*}{Definition}  
\newenvironment{definition*}{\begin{itdefinition*}\rm}{\end{itdefinition*}}  \newtheorem{lemma}{Lemma}  
\newtheorem{theorem}{Theorem}  
\newtheorem*{theorem*}{Theorem}  
\newtheorem{itremark}{Remark}  
  
\newtheorem*{itremark*}{Remark}  
\newtheorem*{corollary*}{Corollary} 
\newenvironment{remark*}{\begin{itremark*}\rm}{\end{itremark*}}  
\newtheorem{prop}{Proposition}  
\newtheorem{itassum}{Assumption}
  
\newtheorem{corollary}{Corollary}  
\newtheorem*{itexample*}{Example}  
\newenvironment{example}{\begin{itexample*}\rm}{\end{itexample*}}

\title[Large Deviations for Partition Functions]{Large Deviations for Partition Functions of Directed Polymers and Some Other Models in an IID Field}
\author{Iddo Ben-Ari}
\address{Department of Mathematics, University of California - Irvine, Irvine, CA 92697, USA}
\ifpdf
\email{\href{mailto:ibenari@math.uci.edu}{ibenari@math.uci.edu}}
\else
\email{ibenari@math.uci.edu}
\fi
\begin{document} 
\begin{abstract}
Consider the partition function of a directed polymer in dimension $d\ge1$ in an IID field. We assume that both  tails of the  negative and the positive part of the field are at least as light as exponential. It is a well-known fact that the  free energy of the polymer is equal to a deterministic constant for almost every realization of the field and that the upper tail of the large deviations is exponential. The lower tail of the large deviations is typically lighter than exponential. In this paper we provide a method to obtain estimates on the rate of decay of the lower tail of the large deviations, which are sharp up to multiplicative constants. As a consequence, we show that the lower tail of the large deviations exhibits three regimes, determined according to the tail of the negative part of the field. Our method is simple to apply and can be used to cover other oriented and non-oriented models including first/last-passage percolation and the parabolic Anderson model.  
\end{abstract}
\maketitle
\date{\today}
\section{Introduction and Statement of Results} 
Let  $V \equiv \{V(t,x):(t,x)\in  \Z_+\times \Z^d\}$ denote an IID field under a probability measure $Q$. We will denote the corresponding expectation operator by $Q$ as well. We will always assume that 
\begin{enumerate}
\item[(AS0)] $V(0,0)$ is non-degenerate;
\item[(AS1)] $Q(V(0,0))=0$;
\item[(AS2)] There exists some constant $\ol{\eta}>0$ such that $Q(e^{\eta V(0,0)})<\infty$ for all $|\eta|<\ol{\eta}$. 
\end{enumerate}
 We remark that the  assumption (AS1) was made only for convenience and does not affect the generality of the results. \par
 Let $|\cdot|$ denote the $l^1$-norm on $\Z^d$, that is the sum of the absolute values of the coordinates. We let $\gamma$ denote a simple symmetric nearest neighbor random walk path on $\Z^d$. In other words, $\gamma:\Z_+\to \Z^d$,  satisfying $|\gamma(t+1)-\gamma(t)|=1$ for all $t\ge 0$. For $x\in \Z^d$, let $P_x$ denote the probability measure induced by the random walk with $P_x(\gamma(0)=x)=1$.  Let $E_x$ denote the corresponding expectation. Define the partition function  $Z(T)$  by letting:  
 $$Z(T)=E_0 e^{H_{\gamma}(T)},\text{ where }H_{\gamma}(T)=\sum_{t=0}^{T-1} V(t,\gamma(t))).$$ 
  Below, we will usually omit the dependence on $\gamma$ and write $H(T)$ meaning  $H_{\gamma}(T)$. Being an expectation of an exponential function, the essential  contribution to $Z(T)$ is  from paths maximizing $H(T)$. Let $\zeta(T)=\sup_{\gamma}H_{\gamma}(T)$, the supremum taken over all paths $\gamma$ with $\gamma(0)=0$. Thus, $Z(T)$ can be thought of as a ``penalized'' version of $e^{\zeta(T)}$. Due this observation, there is a complete analogy between the behavior of $Z(T)$ and of $\zeta(T)$, at least from the point of view of the results below and all are also valid for $\zeta$ with the appropriate minor changes. We remark that $\zeta$ is a model of oriented last-passage site percolation.  For the purpose of making this presentation more simple, we have chosen to discuss $\zeta$ rather than $Z$.\par
 For positive functions $q,r:\R_+\to (0,\infty)$ or $q,r:\Z_+\to \R_+$,  we say that  $q\sim r$  if
 $$0< \liminf_{t\to\infty}  \frac{q(t)}{r(t)}\le \limsup_{t\to\infty} \frac{q(t)}{r(t)}<\infty.$$  
 Clearly,  $\sim$ is an equivalence relation.
A fundamental result is the following: 
 \begin{theorem} ~
\label{th:subadditive}
 \begin{enumerate}
 \item \label{subadd1}
 There exists a constant $\lambda\in [0,\infty)$ such that 
  $$\lambda = \liminf_{T\to\infty} \frac 1 T \ln Z(T)=\limsup_{T\to\infty} \frac 1T \ln Z(T),~Q\text{-almost surely.}$$ 
 \item \label{subadd2} There exists $\epsilon_0 \in (0,\infty]$ such that for every $\epsilon \in (0,\epsilon_0)$, 
$$ -\ln Q(Z(T)\ge e^{(\lambda+\epsilon)T})\sim T.$$
 \end{enumerate} 
\end{theorem}
 Note that $\lambda \ge 0$ due to (AS1).  
 The proof of the theorem is essentially due to subadditive arguments. For a proof of part (\ref{subadd1}), we refer the reader to \cite[Proposition 1.5]{CSY}, where also (\ref{subadd2}) was proved under additional assumptions on the distribution of $V(0,0)$.  For a proof of part (\ref{subadd2}), we refer the reader to \cite[Theorem 2.11]{CMS}, where the analogous  result for the Parabolic Anderson Model was established. This was done through discretization, which makes the proof almost identical to the present model. The proof is based on a percolation argument. The analogue of Theorem \ref{th:subadditive} for $\zeta$ is the following: There exists a constant $\mu \in [0,\infty)$ such that $\lim_{T\to\infty} \frac 1T \zeta(T)=\mu,~Q$-almost surely, and $-\frac 1T \ln Q(\zeta(T)>(\mu+\epsilon)T)\sim 1$.
\par  
 In this paper we study the lower tail of the large deviations of $Z(T)$, namely the behavior of  $Q(Z(T)\le e^{(\lambda-\epsilon)T})$ for $\epsilon>0$. For every $\epsilon>0$ define a function  $R_{\epsilon}:\Z_+\to [-\infty,0]$ by letting 
 $$R_{\epsilon}(T)=-\ln Q(Z(T)\le e^{(\lambda-\epsilon)T}).$$
 The function $R_{\epsilon}$ will be called ``the rate''. Similarly, we let $R_{\epsilon}^{\zeta}(T)=-\ln Q(\zeta(T)\le (\mu-\epsilon)T)$. The main goal is to find the functional dependence of $R_{\epsilon}$ on the distribution of $V(0,0)$. Intuitively, the difference between the upper tail of the large deviations of Theorem \ref{th:subadditive}-(\ref{subadd2}) and the lower tail of the large deviations can be explained as follows: In order for $\zeta(T)$ to be bigger than $(\mu+\epsilon)T$, we need $H_{\gamma}(T)\ge (\mu+\epsilon)T$ for one path $\gamma$, but in order for it to be smaller than $(\mu-\epsilon)T$, we need $H_{\gamma}(T)<(\mu-\epsilon)T$ for all paths $\gamma$. Of course, the latter event is typically significantly less probable. Therefore, one may expect that for some fields, the rate will be of an order larger than $T$. Other models known to exhibit asymmetry between the upper and lower tails of the  large deviations include (non-oriented) first-passage percolation  \cite[Theorem 4.3]{kesten} \cite{CZ}, length of the longest increasing subsequence in a random permutation \cite{aldous}, and the longest increasing sequence of random samples on the unit square \cite{zeitouni}.  We now sketch a mathematical argument that can be used to prove such an asymmetry for $\zeta$. For reasons soon to become clear, it will be called ``the independence argument''. Let $c$ denote a positive constant that may vary from line to line. At the core lies the observation that given $\epsilon>0$, one can find a cube $C\subset Z^d$ centered at the origin, with side-length depending on $\epsilon$, such that the supremum of $H_{\gamma}(T)$ over all paths $\gamma$ with $\gamma(0)=0$ and $\gamma(t)\in C$ for all $t<T$ is bigger than $(\mu-\epsilon)T$ with probability bounded below by $1-e^{-cT}$. Roughly, this is proved by ``navigating'' paths towards the origin while controlling the probability using the FKG inequality. Call $C$ ``good'' if this event occurs.  Suppose we have $N$ disjoint translates of $C$ and for each one we consider the shifted version of $\zeta(T)$, that is the supremum taken over all paths starting from the shifted  center. Due to independence, the probability that some cube is good is bounded below by $1-e^{-cNT}$. Let $M:\Z_+\to\Z_+$ be a function satisfying $M(T)\underset{T\to\infty} {\nearrow} \infty$ and $M(T)\le T/2$. At time $M(T)$ there is an order of $M(T)^d$ points $x\in \Z^d$ for which $\gamma(M(T))=x$ for some path $\gamma$ with $\gamma(0)=0$.  A certain proportion of these points, depending on the size of $C$, can be declared as centers of disjoint translates of $C$. Suppose for the moment that $V(0,0)$ is bounded from below by, say,  $-1$. In this case, $H_{\gamma}(M(T))\ge -M(T)$ for all paths $\gamma$ with $\gamma(0)=0$. As a result,   the probability that $\zeta(T)\ge (\mu-\epsilon-M(T)/T)T$ is bounded below by $1-e^{-c M(T)^d T}$. Choosing $M(T)=\f{\epsilon T}$, we immediately see that $R^{\zeta}_{2\epsilon}(T)\ge c T^{1+d}$.  This type of argument was used in \cite{CZ} to prove the corresponding result for (non-oriented) first-passage percolation in a nonnegative field.  An upper bound on the rate is simpler. Continuing with the same example, suppose that $Q(V(0,0)=-1)>0$ and let $M(T)$ be as above. Consider now the event that $V(t,x)=-1$ for all $t<M(T)$ and all $|x|\le t$. This event has probability bounded below by $e^{-cT^{1+d}}$. Therefore, it easily follows from Theorem \ref{th:subadditive}-(\ref{subadd2}) and the FKG inequality that $\zeta(T) \le e^{-\f{\epsilon T}} e^{(\mu+\epsilon/2)(T-M(T))}\le e^{(\mu-\epsilon/2)T}$ with probability bounded below by $e^{-c T^{1+d}}(1-e^{-c(T-M(T))})^{(1+M(T))^d}\sim e^{-\frac c2 T^{1+d}}$. Thus,  $R^{\zeta}_{\epsilon/2}\le  cT^{1+d}$.
\par 
  When $V$ is unbounded from below, the contribution of the paths near the beginning may drastically affect the rate.  This situation was first treated in \cite{CGM}, for a model of oriented last-passage bond percolation, as well as for a (non-oriented) first-passage percolation model. We refer to the the function $x\to -\ln Q(-V(0,0)>x)$ as the ``negative tail''. The main results of the above paper are a perturbation result giving a necessary and sufficient condition on the negative tail  to guarantee that $R^{\zeta}_{\epsilon}(T)\sim T^{1+d}$ (Corollary \ref{cor:fast} below) and an estimate for the rate in the Gaussian case (Corollary \ref{cor:gaussian} below) in one dimension. The lower bound on the rate in \cite{CGM} was obtained through a certain construction of paths near the beginning. This construction depends on the realization of the field and therefore leads to  an elaborate process of choosing realizations, controlling their probabilities and matching corresponding paths. Due to its nature, this method requires an a-priori estimate of the rate and is hard to apply for more general fields.
\par In this paper we develop a different approach based on a universal construction, which reduces the estimation of the rate to an optimization problem and allows us to obtain estimates for the rate in terms of the negative tail for a large class of fields. As our results show, one can summarize the dependence of the rate on the negative tail as follows:  
\begin{itemize}
\item When the negative tail is ``sufficiently large", then the rate is comparable to it (Theorem \ref{th:slow}-(\ref{slow1})); 
\item When the negative tail is ``sufficiently small", then $R_{\epsilon}(T)\sim T^{1+d}$ (Corollary \ref{cor:fast});
\item Transition. The rate is $o(\min(T G(T),T^{1+d}))$ (Corollary \ref{cor:transition}, for example).   
\end{itemize} 
We begin with the following simple result:
\begin{prop}
\label{pr:expo} 
 Assume that (AS0)-(AS2) hold and, in addition, $-\ln Q( -V(0,0)>x)\sim x$.  Then, $R_{\epsilon}(T)\sim T$. 
\end{prop} 
 For our main results, further assumptions on the negative tail are required. Unless otherwise stated, in addition to (AS0)-(AS2), below we will always assume the following:
\begin{enumerate}
\item[(AS3)] There exists a constant $\ol{x}>0$ and a continuous, strictly increasing function $G:\R_+\to \R_+$ such that 
 \begin{align*}
   &\lim_{t\to\infty} G(x)=\infty;\\
   &Q(-V(0,0)> x) = e^{-x G(x)},\text{ for all }x\ge \ol{x}.
 \end{align*}
\end{enumerate}
 We note that there is no loss of generality assuming that $G(0)=0$. Therefore it follows that $G$ has a continuous, strictly increasing inverse, $G^{inv}:[0,\infty)\to \R_+$ with the properties: 
 $$G^{inv}(0)=0,\text{ and }\lim_{x\to\infty} G^{inv}(x)=\infty.$$ 
 We need some additional notation. Let 
 \begin{align*}
 f(x) &= \frac{G(x)}{x^d},~x>0;\text{ and let }\\
 F(z) &= z^{1/d} \int_{G^{inv}(1)}^{G^{inv}(z)} G^{-1/d}(x)dx,~z\ge G^{inv}(1).
\end{align*}  
 We have chosen to work with monotone $f$. We split the results according to whether $f$ is non-increasing or non-decreasing. In all results below, $\epsilon$ is assumed to be any positive constant. We begin with the case that $f$ is non-increasing. In terms of the negative tail, this corresponds to the case where it is not larger than $O(T^{1+d})$. 
 \begin{theorem}
 \label{th:slow}
 Suppose that $f$ is non-increasing.
 Let 
 $$ \gamma = \limsup_{y\to\infty}  \frac{F(G(y))}{y}= \limsup_{y\to\infty} f^{1/d}(y)\int_{G^{inv}(1)}^y G^{-1/d}(x)dx.$$
  \begin{enumerate}
 \item\label{slow1}  If $\gamma<\infty$ then $R_{\epsilon}(T)\sim T G(T)$. 
 \item\label{slow2} If $\gamma=\infty$, then there exists a constant $C>0$, depending only on $\epsilon$ and the distribution of $V(0,0)$ such that  
 for every $\delta >0$
 $$\limsup_{T\to\infty} \frac{R_{\epsilon}(T)}{T G(\delta T)}\le C.$$
\end{enumerate} 
\end{theorem} 
 Under stronger requirements on $f$, we obtain a necessary and sufficient condition: 
\begin{corollary}
\label{cor:slowuseful}
 Suppose that $f$ is convex, $\lim_{x\to\infty} f(x)=0$ and that the limit
  $$\rho =\lim_{x\to\infty}-\frac{\frac{d}{dx}\ln f(x)}{\frac{d}{dx} \ln x}=\lim_{x\to\infty}-\frac{x f'(x)}{f(x)}$$
 exits.  
\begin{enumerate}
\item If $\rho>0$, then $R_{\epsilon}(T)\sim T G(T)$; 
\item If $\rho=0$, then $R_{\epsilon}(T)=o(T G(T))$. 
\end{enumerate}
\end{corollary}
The proof of the corollary is given at the end of Section \ref{sec:proofs}. 
As a concrete example we have 
\begin{example}~
\begin{enumerate}
\item Suppose that $G(x)= x^{\alpha}$, for $\alpha\in [0,d)$. 
 Then $R_{\epsilon}(T)\sim T^{1+\alpha}$.  
\item Suppose that $G(x)= x^d e^{(-\ln x)^{\beta}}$ for $\beta \in [0,1)$. 
 Then $R_{\epsilon}(T)=o(T G(T))$. 
\end{enumerate}
\end{example}
 We now move the the case where $f$ is non-decreasing. That is, the negative tail is not smaller than  $O(T^{1+d})$. 
\begin{theorem}
\label{th:fast}
 Suppose that $f$ is non-decreasing. Let $\eta:\Z_+\to \R_+$ be such that  
$F(\eta(T))\sim T$. Then $R_{\epsilon}(T) \sim T \eta(T)$. 
\end{theorem}
The theorem has two immediate corollaries:  
\begin{corollary}
\label{cor:fast}
 Suppose that $f$ is non-decreasing. Then 
 $$R_{\epsilon}(T) \sim  T^{1+d}\text{ if and only if }\int^{\infty} G^{-1/d}(x)dx<\infty.$$
\end{corollary}
\begin{proof}
 If the integral converges, then we may take $\eta(T)=T^d$. On the other hand, if the integral diverges, the condition $F(\eta(T))\sim T$ implies that $\eta(T)=o(T)$. Therefore $R_{\epsilon}(T)=o(T^{1+d})$.
\end{proof}
 Corollary \ref{cor:fast} was first proved in \cite{CGM}.
\begin{corollary}
 \label{cor:transition}
 Suppose that $f$ is non-decreasing and bounded. Then 
 $$R_{\epsilon}(T) \sim \frac{T^{1+d}}{\ln^d T}.$$
\end{corollary} 
\begin{proof}
 Since $f$ is non-decreasing and bounded, $G(x) \sim x^d$. In particular, $G^{inv}(x)\sim x^{1/d}$.  Therefore, $F(z) \sim z^{1/d} \ln z$. Thus, the condition of Theorem \ref{th:fast} is satisfied with $\eta(T)=T^{d}/\ln^d T$.
\end{proof}
 Combining Corollary \ref{cor:transition} and Theorem \ref{th:slow}-(\ref{slow1}) we obtain 
 \begin{corollary}
 \label{cor:gaussian}
   Suppose that $V$ is Gaussian. Then 
 $$R_{\epsilon}(T) \sim \begin{cases}  T^{2}/\ln T & d=1;\\  T^2 & d\ge 2.
 \end{cases} $$
 \end{corollary}
\begin{proof}
Since  $V$ is Gaussian,  $G(x)\sim x$. Equivalently,  $f(x)\sim x^{1-d}$. Therefore, when $d=1$ we may apply Corollary \ref{cor:transition}. When $d\ge 2$, $F(G(t))\sim t^{1/d} \ln t = o(t)$. Therefore $\gamma=0$ and it follows from Theorem \ref{th:slow}-(\ref{slow1}) that $R_{\epsilon}(T)\sim T^2$.
\end{proof}
  In \cite{CH} it was proved using concentration inequalities that when $d\ge 3$ and $V(0,0)\sim N(0,\beta^2)$,  for some sufficiently small  $|\beta|$, then $\liminf_{T\to\infty} \frac{R_{\epsilon}(T)}{T^2}>0$. In one dimension, the corollary was proved in \cite{CGM}.\par 
 We conclude this section with an explanation of our method. We begin with the lower bound on the rate. The main idea is to construct a set of paths $\tilde \Gamma$, all starting from the origin, which is combinatorially simple and at the same time rich enough to allow that for all $t\ge 0$, the mapping $\gamma\to \gamma(t)$ from $\tilde \Gamma$ to $\Z^d$ has range of the order of $t^d$. Recall that $M:\Z_+\to \Z_+$ is a function satisfying $M(T)\underset{T\to\infty}{\nearrow} \infty$ and $M(T)\le T/2$. Call a path  $\gamma\in \tilde \Gamma$ ``open'' if $H_{\gamma}(M(T))\ge -\epsilon T$. Let $E$ denote the event that a proportion of $r\in (0,1)$ of the paths in $\tilde \Gamma$ is open. Let $(t,x)\in \Z_+\times \Z^d$. When $\gamma(t)=x$ for some $\gamma\in \tilde \Gamma$ we say that $\gamma$ visits $x$ at time $t$. The basic idea of the construction is that if a large proportion of paths visit $x$ at $t$,  then when $V(t,x)$  attains a large negative value, this affects all of them ``free of charge'', probability-wise. This cannot be completely avoided, as all paths have to begin from the origin. However, we can minimize the damage by requiring that for each time $t$, all points visited at time $t$ are visited by a comparable proportion of paths. This uniformity  leads almost immediately to simple lower bound on the probability of $E$ derived directly from upper bounds on the moment generating function of $-V(0,0)$ through the Markov inequality. Denote  this lower bound by $1-e^{-J(T)}$. On $E$, we may  use the open paths as channels leading from the origin to an order of $M(T)^d$ centers of disjoint translates of $C$, allowing for the independence argument to apply. Since this involves only some of the paths starting from the origin, it follows that the probability that $\zeta(T)\ge (\lambda-2\epsilon)T$ is bounded below by $(1-e^{-cM(T)^dT})(1-e^{-J(T)})\ge 1-e^{-\frac 12\min(cM(T)^dT,J(T))}$. Thus, $R_{2\epsilon}^{\zeta}(T)\ge \frac 12  \min(cM(T)^d T, J(T))$. The rest is optimization. The upper bound is an improvement of the method presented above when $V(0,0)$ is bounded from below.\par 
\section{Proofs}
\label{sec:proofs}
 We begin with some additional notation. Let $L_t$ denote the set of points $x\in\Z^d$ for which there exist a path $\gamma$, with $\gamma(0)=0,\gamma(t)=x$. Clearly, $|L_t|\le (1+t)^d$.  For a path $\gamma$, we let  $H_{\gamma}(t_1,t_2)=\sum_{t=t_1}^{t_2-1} V(t,\gamma(t-t_1))$. Thus, $H_{\gamma}(T)=H_{\gamma}(0,T)$. We may  sometimes omit the dependence on $\gamma$. \\
 Let $\eta:\Z_+\to \R_+$ and $M:\Z_+\to \Z_+$. We define $I_{\eta}^M:\Z_+\to \R_+$ by letting
 $$I_{\eta}^{M}(T)=\sum_{t=0}^{M(T)-1} G^{inv}(\frac{\eta(T)}{(1+t)^d}).$$
  We also define the function $F_{\eta}^M:\Z_+\to \R_+$ by letting
   $$F_{\eta}^M(T)=\eta(T)^{1/d} \int_{G^{inv}(\eta(T)/M(T)^d)}^{G^{inv}(\eta(T))} G^{-1/d}(x)dx.$$
    The next result allows us to replace series with integrals. 
\begin{lemma}
\label{lem:lll}
 Suppose that $f$ is monotone. Then, there exists a constant $C_0>0$ depending only on $f$ and $d$ such that  for all $T$ satisfying $1\le M(T)^d\le \eta(T)$  
   $$F_{\eta}^M(T)+\eta^{1/d}(T)\Delta(T)\le I_{\eta}^M(T) \le F_{\eta}^M(T)+C_0 \eta^{1/d}(T),$$
 where $$\Delta(T)=f^{-1/d}(G^{inv}(\frac{\eta(T)}{M(T)^d}))-f^{-1/d}(G^{inv}(\eta(T))).$$
\end{lemma}
\begin{proof}
  To simplify notation we write $I(T)$ instead of $I_{\eta}^M(T)$. 
 Let $L(T)=\int_1^{M(T)} G^{inv}(\frac{\eta(T)}{y^d})dy$. 
 Clearly, 
  \begin{equation}\label{eq:LI} L(T) \le I(T) \le G^{inv}(\eta(T))+L(T).
   \end{equation}
 By changing variables to $u= y/\eta^{1/d}(T)$ we obtain 
  $$L(T) = \eta^{1/d}(T)\int_{\eta(T)^{-1/d}}^{M(T)/\eta^{1/d}(T)} G^{inv}(u^{-d}) du.$$
   We perform a second change of variables. Let $x=G^{inv}(u^{-d})$. Then  $u=G^{-1/d}(x)$ and we have 
\begin{align*}
 \nonumber
G^{inv}(u^{-d})\partial u &= x d G^{-1/d}(x) =-\frac {1}{d}  G^{-(1+1/d)}(x)G'(x)\partial x\\
 \nonumber 
 & =  -\frac{1}{d} \frac{x}{x^{1+d}f^{1+1/d}(x)}\left(d x^{d-1} f(x)+x^d f'(x)\right) \partial x\\
  \label{eq:rho(u)du}
  & = -\left(\frac{1}{xf^{1/d}(x)}+\frac{f'(x)}{df^{1+1/d}(x)}\right)\partial x=
   -G^{-1/d}(x)\partial x+ \partial f^{-1/d}(x),
\end{align*}
 where we have used $\partial$ to denote the differential, in order to avoid confusion with the dimension $d$.
  Therefore
   $$L(T)=\eta^{1/d}(T)\left(\int_{G^{inv}(\frac{\eta(T)}{M(T)^d})}^{G^{inv}(\eta(T))}G^{-1/d}(x)dx+\Delta(T)\right),$$
       Recall that $f$ is assumed to be monotone. When $f$ is non-decreasing, $f^{-1/d}$ is non-increasing and is therefore bounded. When $f$ is non-increasing, $f^{-1/d}$ is non-decreasing and therefore $\Delta(T)\le 0$. In particular, there exists a constant $C_0\ge 0$,  depending only on $f$ and $d$ such that $\Delta(T)\le C_0$. 
    The second inequality in \eqref{eq:LI} gives
     $$I(T)\le F_{\eta}^M(T)+C_0 \eta^{1/d}(T),$$
      proving the second inequality in the lemma. To conclude the proof, note that for every $z> 0$, $z^{1/d} = G^{1/d}(G^{inv}(z))=G^{inv}(z) f^{1/d}(G^{inv}(z))$, therefore $f^{-1/d}(G^{inv}(z))=z^{-1/d}G^{inv}(z)$. This gives
    $$ L(T)=F_{\eta}^M(T)+M(T) G^{inv}(\frac{\eta(T)}{M(T)^d})-G^{inv}(\eta(T)).$$
     Thus, the first inequality in \eqref{eq:LI} gives
   $$ I(T) \ge F_{\eta}^M(T)-G^{inv}(\eta(T)).$$
 \end{proof}
\subsection{Lower Bound}
\label{sec:lowerbound}
Our main result is the following: 
 \begin{prop}
\label{pr:ubprop}
 Suppose that there exists a constant $C>0$ and  $\eta:\Z_+\to\R_+$ such that the following conditions hold:  
 \begin{enumerate}
 \item\label{etalet} $\limsup_{T\to\infty} \frac{\eta^{1/d}(T)}{T}<C$;
 \item $\limsup_{T\to\infty} \frac{F(\eta(T))}{T}<C$. 
 \end{enumerate}
  Then 
 $$\liminf_{T\to\infty} \frac{R_{\epsilon}(T)}{T \eta(T)}>0.$$
 \end{prop}
 The proof of the proposition will be preceded by a sequence of lemmas.  We begin with an estimate on the moment generating function of $-V(0,0)$.  
 \begin{lemma}
 There exists a constant $\eta_0>0$ depending only on the distribution of $V(0,0)$ such that for all $\eta'>\eta_0$,
 \label{lem:expmom2}
 \begin{equation*}
 Q(e^{-\eta' V(0,0)})\le e^{2\eta' G^{inv}(2\eta')}\end{equation*}
\end{lemma}
\begin{proof}
 For every $\eta',\rho>0$, 
 $$Q(e^{-\eta' V(0,0)})=Q\left(e^{-\eta ' V(0,0)}\left(\ch_{\{-V(0,0)\le \rho\}}+\ch_{\{-V(0,0)>\rho\}}\right)\right).$$
 Let $\eta_1$ be such that $G^{inv}(2\eta)\ge \ol{x}$ for all $\eta \ge \eta_1$. Let $\eta' \ge \eta_1$ and let $\rho=G^{inv}(2\eta')$. Then, $\rho \ge \ol{x}$ and it follows that 
 \begin{align*}
 Q(e^{-\eta' V(0,0)})&\le e^{\eta' \rho}+1+\eta' \int_{\rho}^{\infty}Q(-V(0,0)>x) e^{\eta' x}dx\\
 & \le e^{\eta' G^{inv}(2\eta')}+1 +\eta' \int_{\rho}^{\infty} e^{-t\eta'(\frac{G(x)}{\eta'}-1)}dx.
\end{align*}
 For $x\ge \rho$,  $\frac{G(x)}{\eta'}-1\ge \frac{G(\rho)}{\eta'}-1=1$. Thus,  
 $$Q(e^{-\eta' V(0,0)})\le e^{\eta' G^{inv}(2\eta')}+1+\eta' \int_{\rho}^{\infty} e^{-x \eta'}dx = e^{\eta' G^{inv}(2\eta')}+1+e^{-\eta' G^{inv}(2\eta')}.$$
 Since $\lim_{x\to\infty} G^{inv}(x)=\infty$, the claim follows by choosing $\eta_0$ large enough.   
 \end{proof}
 We now construct the set of  paths discussed in the introduction. Below we write $\gamma_{t,x}$ meaning some path with the property $\gamma(t)=x$. If $|z-x|=1$, we write $\gamma_{t,x}\oplus z$ for the path $\gamma$ which coincides with $\gamma_{t,x}$ up to time $t$ and satisfies $\gamma(t+1)=z$. For every $x\in \Z^d,c\in \{1,\dots,d\}$ and $r=\pm 1$, we let $x^{c,r}=(x^{c,r}_1,\dots,x^{c,r}_d)\in \Z^d$ satisfy $x^{c,r}_c=x_c+r$ and $x^{c,r}_k = x_k$ for all  $k\ne c$.
 Let $S_0=\{0\}$, $c_0=1$ and $\tilde \Gamma_0=\{\gamma_{0,0}\}$, for some path $\gamma_{0,0}$. We continue inductively: 
 \begin{enumerate}
 \item\label{big_cycle}  Let $l_t=\min\{x_{c_t}:x\in S_t\}$.
 \item\label{small_cycle} Let $x\in S_t$. If $x_{c_t}=l_t$, we let $\gamma_{t+1,x^{c_t,-1}}=\gamma_{t,x}\oplus \{x^{c_t,-1}\}$ and $\gamma_{t+1,x^{c_t,+1}}=\gamma_{t,x}\oplus \{x^{c_t,+1}\}$. 
 Otherwise, let $x^*=x^{c_t,\sign(x_{c_t}-l_t)}$ and set $\gamma_{t,x^*}=\gamma_{t,x}\oplus \{x^*\}$. We set $S_{t+1}=\{x^{c_t,\pm 1}:x \in S_t\}$ and let
 $\tilde \Gamma_{t+1}=\{\gamma_{t+1,x}:x\in S_{t+1}\}$. 
  If $l_t+3\le\max\{x_{c_t}:x\in S_{t+1}\}$, then we set $l_{t+1}=l_t+3$, $c_{t+1}=c_t$ and return to step (\ref{small_cycle}), starting from time $t+1$. Otherwise, we let $c_{t+1}=(c_t \mod d)+1$ and return to step (\ref{big_cycle}), starting from time $t+1$. 
\end{enumerate}
 Figure 1 illustrates the construction of $\tilde \Gamma_t$ in one dimension. The horizontal axis is the time axis starting from $t=0$ on the left. The vertical axis is the space axis with $x=0$ in the center. For each time $t$, $S_t$ is represented by the round nodes on the corresponding vertical line. The large  nodes represent the value of $l_t$. The paths in $\tilde \Gamma_t$ are obtained by following the solid lines from left to right from time $0$ to time $t$.\\
 \begin{figure}[h]
 \includegraphics{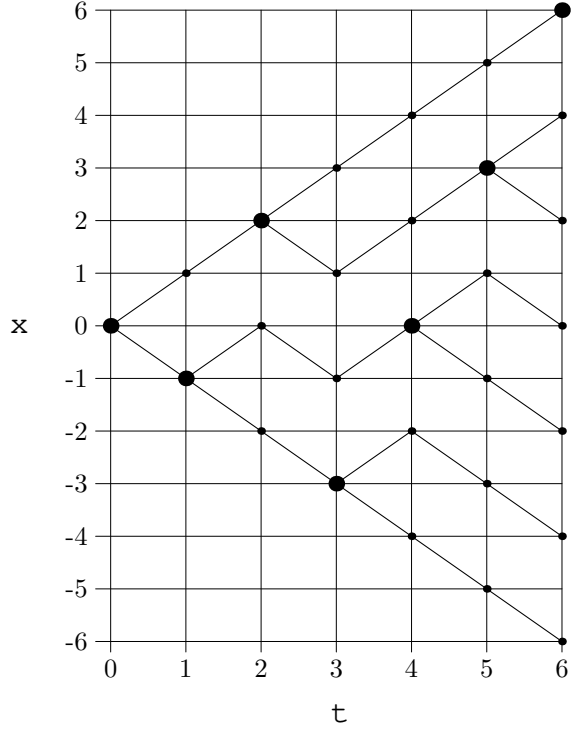}
\caption{$\tilde \Gamma$ in one dimension}
\end{figure}
 For $t'\ge t$, we let $n_{t'}(t,x)=|\{\gamma\in \tilde\Gamma_{t'}:\gamma(t)=x\}|$.  Set $\sigma_0=0$ and let $\sigma_{k+1}=\min\{t\ge \sigma_k:c_t=1\}$. Starting from time $\sigma_k$, in each step of the construction we double one of the hyperplanes of $S_{\sigma_k}$. We begin with all hyperplanes orthogonal to $(1,0,\dots,0)$ then all hyperplanes orthogonal to $(0,1,0,\dots)$, until we finish with all hyperplanes orthogonal to $(0,\dots,1)$. We repeat the process again at time $\sigma_{k+1}$, doubling all hyperplanes of $S_{\sigma_{k+1}}$. Between time $\sigma_k$ and $\sigma_{k+1}$ we double each of the side lengths of $S_{\sigma_k}$. Since $S_{0}$ is a cube of side length $1$, it follows that $S_{\sigma_k}$ is a cube of side length $2^k$. In addition, this shows that $\sigma_{k+1}-\sigma_k=d2^k$. Thus, $\sigma_k = d \sum_{j=0}^{k-1} 2^j = d (2^k-1)$. We also note that given some $s$ and $j$ such that $\sigma_j \le s < \sigma_{j+1}$, any point in $S_s$ will be split into at most $2^d$ points by time $\sigma_{j+1}$. Fix now $t'\ge t$ and let $k'$ be such that $\sigma_{k'}\le t'<\sigma_{k'+1}$.  Thus,   $n_{t'}(t,x)$ is bounded above by $2^{d(1+k'-k)}$. Now $t' \ge \sigma_{k'}=d(2^{k'}-1)$ and $t < \sigma_{k+1}=d(2^{k+1}-1)$. 
 Therefore, $k'-k \le (\log_2((t'+d)/d)-\log_2((t+d)/d)+1)= \log_2 ((t'+d)/(t+d))+1$. Hence, $n_{t'}(t,x) \le 2^{2d} (\frac{t'+d}{t+d})^d$. In addition, $|S_t|\ge |S_{\sigma_k}|=2^{d k}\ge 2^{d (\log_2((t+d)/d)-1)}=(2d)^{-d}(t+d)^d$. We summarize what we have proved in the following lemma: 
\begin{lemma}~
\label{lem:ntx}
 \begin{enumerate}
 \item If $t'\ge 1$, then 
 $n_{t'}(t,x) \le (4+4d)^d (\frac{t'}{1+t})^d$; 
 \item For all $t\ge 0$, $|S_t|\ge (2d)^{-d}(t+d)^d$.  
  \item Combining the above two estimates  we obtain: 
 $$\frac{n_{t'}(t,x)}{|S_{t'}|}\le \frac{(4d)^{2d}}{(1+t)^d}.$$
 \end{enumerate}
 \end{lemma}
We need the following lemma, which is an adaptation of the results of \cite[Section 2]{CGM}. Its proof is given in the appendix.\par
 \begin{lemma}
 \label{lem:channels}
 Assume that (AS0)-(AS2) hold. Then, for every $\epsilon >0$ there exist $W(\epsilon)\in\N$ and $c(\epsilon)>0$ such that  
$$Q( E_0 [\exp(H(T))\ch_{\{\max_{s\in \{0,\dots,T\}}|\gamma(s)|\le W\}}] \ge e^{(\lambda -\epsilon) T})\ge 1-e^{-c T},$$
 for all $T$ sufficiently large. 
\end{lemma}
  We now fix some $\epsilon>0$ and some $M:\Z_+\to \Z_+$. We write $\tilde\Gamma$ as a shorthand notation for  $\tilde \Gamma_{M(T)}$,   $n(t,x)$ as a shorthand notation for $n_{M(T)}(t,x)$. In addition, for $A\subset \tilde \Gamma$,  we let $n_A(t,x)=|\{\gamma\in A:\gamma(t)=x\}|$. For $x\in S_{M(T)}$ we let $\tilde \gamma_x$ denote the unique path in $\tilde \Gamma$ with $\tilde \gamma(M(T))=x$.  Let  $r(\epsilon)=1-\frac{1}{4(8W(\epsilon))^d}$, where $W(\epsilon)$ is as in Lemma \ref{lem:channels}. We define  
 $$G_1=\{x\in S_{M(T)}:H_{\tilde \gamma_{x}}(M(T))\ge -\epsilon T\}\text{ and }E=\{|G_1|\ge r|S_{M(T)}|\}.$$
 We have 
\begin{lemma}
 \label{lem:markov}
  Let $K=4(4d)^{2d}$ and let ${\tilde \eta}:\Z_+\to \R_+$. Suppose that 
 \begin{enumerate}
 \item $\lim_{T\to\infty} M(T)=\infty$;
 \item ${\tilde \eta}(T)\ge 2\eta_0 M(T)^d$, for all $T$ sufficiently large. 
 \end{enumerate}
 Then there exists a constant $T_0\in \Z_+$, depending only on $r$ and $M$ such that  
\begin{equation*}
  Q(E^c)\le\exp \left(2^d M(T)^d \ln 2 + {\tilde \eta}(T) I_{{\tilde \eta}}^M(T)-\frac{(1-r)\epsilon }{K}T{\tilde \eta}(T)\right),\text{ for all }T\ge T_0. 
 \end{equation*}
\end{lemma}
\begin{proof}
 Clearly, 
\begin{align*}
  Q(E^c) &\le \sum\limits_{A\subset \tilde \Gamma,|A|=\f{(1-r)|S_{M(T)}|}}
 Q(\underset{\gamma\in A}{\cap}\{-H_{\gamma}(M(T))> \epsilon T\})\\ 
 &\le
 \sum_{A\subset\tilde \Gamma,|A|=\f{(1-r)|S_{M(T)}|}}
 Q( \sum_{\gamma\in A} -H_{\gamma}(M(T))> \epsilon |A| T).
\end{align*}
We note that 
 $-\sum_{\gamma\in A} H_{\gamma}(M(T)) =-\sum_{t=0}^{M(T)-1}\sum_x n_A(t,x) V(t,x)$. Since $|A|=\f{|S_{M(T)}|(1-r)}$ and $\lim_{T\to\infty} M(T)=\infty$, there exists a positive constant $T_0$ such that  $|A|\ge \frac{|S_{M(T)|}(1-r)}{2}$ for all $T\ge T_0$. Therefore,  
\begin{align*}
 \{ -H_{\gamma}(M(T))> \epsilon |A| T \}&= \{-\sum_{t=0}^{M(T)-1}\sum_x \frac{n_A(t,x)}{|S_{M(T)}|} V(t,x)>\frac{ \f{|S_{M(T)}|(1-r)} \epsilon T}{|S_{M(T)}|} \}\\
&\subseteq \{-\sum_{t=0}^{M(T)-1}\sum_x \frac{n_A(t,x)}{S_{M(T)}} V(t,x)>\frac{(1-r)\epsilon T}{2}\}
\end{align*}  
Hence, 
$$Q(E^c) \le \sum_{A\subset \tilde\Gamma,|A|= \f{(1-r)|S_{M(T)}|}}
 Q( -\sum_{t=0}^{M(T)-1}\sum_x \frac{n_A(t,x)}{|S_{M(T)}|} V(t,x) > \frac{(1-r)\epsilon T}{2}).$$
 Using the Markov inequality, 
 $$Q( -\sum_{t=0}^{M(T)-1}\sum_x \frac{n_A(t,x)}{|S_{M(T)}|} V(t,x) >\frac{(1-r)\epsilon T}{2})\le \left(\prod_{t=0}^{M(T)-1} \prod_x Q(e^{-\eta' \frac{n_A(t,x)}{|S_{M(T)}|}V(0,0)})\right)e^{-\frac{\eta' (1-r)\epsilon T}{2}}$$
  for all $\eta'>0$. 
 Since $Q(V(0,0))=0$, Jensen's inequality implies that the mapping $\mu \to Q(e^{-\mu V(0,0)})$ in non-decreasing on $[0,\infty)$. Therefore, we may replace $n_A(t,x)$ on the righthand side above above with the larger number $n(t,x)$ to obtain a looser upper bound for the lefthand side. However, by Lemma \ref{lem:ntx}-(iii), $\frac{n(t,x)}{S_{M(T)}}\le \frac{(4d)^{2d}}{(1+t)^d}=\frac{K/4}{(1+t)^d}$.   Thus, 
\begin{equation}
\label{eq:jensen}
 1\le Q(e^{-\eta' \frac{n_A(t,x)}{|S_{M(T)}|}V(0,0)})\le Q(e^{-\frac{\eta' K/4 }{(1+t)^d}V(0,0)}).
\end{equation}
The righthand side is independent of the choice of $A$. Since the  number of possible choices for $A$ is equal to $\binom{|S_{M(T)}|}{\f{(1-r)|S_{M(T)}|}}<2^{|S_{M(T)}|}\le 2^{(1+M(T))^d}\le 2^{2^d M(T)^d}$, we get  
\begin{align*}
Q(E^c) &\le 2^{2^d M(T)^d} \prod_{t=0}^{M(T)-1} \left(Q(e^{-\frac{\eta' K/4}{(1+t)^d}V(0,0)})\right)^{|S_t|}e^{-\frac{\eta' (1-r)\epsilon T}{2}}\\
 &\le  2^{2^d M(T)^d} \prod_{t=0}^{M(T)-1} \left(Q(e^{-\frac{\eta' K/4}{(1+t)^d}V(0,0)})\right)^{(1+t)^d}e^{-\frac{\eta' (1-r)\epsilon T}{2}},
\end{align*}
 where the second inequality is due to \eqref{eq:jensen}. 
 Let now $\eta'= 2{\tilde \eta}(T) /K$. For all $t\in \{0,\dots,M(T)-1\}$,
 $$\frac{\eta' K/4}{(1+t)^d}=\frac{{\tilde \eta}(T)/2}{(1+t)^d}\ge\frac{{\tilde \eta}(T)/2}{M(T)^d}\ge \eta_0,$$
 where we have used (ii) to obtain the last inequality. It follows from Lemma \ref{lem:expmom2} that 
\begin{align*}
 Q(E^c)&\le \exp (2^d M(T)^d \ln 2+\sum_{t=0}^{M(T)} \frac{{\tilde \eta}(T)}{(1+t)^d} G^{inv}(\frac{{\tilde \eta}(T)}{(1+t)^d})(1+t)^d-\frac{{\tilde \eta}(T)(1-r)\epsilon T}{K})\\
 &\le \exp (2^d M(T)^d \ln 2 + {\tilde \eta}(T) I_{{\tilde \eta}}^M(T)-\frac{(1-r)\epsilon }{K}T{\tilde \eta}(T))
\end{align*}
\end{proof} 
 Once we have obtained control over the contribution near the beginning, we are ready to combine the estimates with the independence argument: 
\begin{lemma}
\label{lem:glueing} 
 Suppose that 
\begin{enumerate}
 \item There exists a function $J:\Z_+\to \R_+$ such that $\ln Q(E^c)\le - J(T)$ for all  sufficiently large $T$; 
 \item $\limsup_{T\to\infty} \frac{\max(\lambda,1)M(T)}{T}<\epsilon$.
 \end{enumerate} 
  Then, there exists a constant $C_{\infty}>0$ depending only on  $\epsilon,d$ and the distribution of $V(0,0)$ such that 
 $$-\ln  Q( Z(T)\le e^{(\lambda-4\epsilon)T})\ge \frac 12 \min(C_{\infty} T M(T)^d,J(T)),$$
  for all sufficiently large $T$. 
\end{lemma}
\begin{proof}
  Let $W$ and $c$ be as in Lemma \ref{lem:channels}. Let $A=\{x\in S_{M(T)}:x=4Wk,~k\in \Z^d\}$.
  Clearly, 
  $$|S_{M(T)}|\ge |A|\ge \frac{|S_{M(T)}|}{(8W)^d} $$
  We let  
 $$G_2=\{x\in A: E_x \left[\exp(H(M(T),T))\ch_{\{\max_{s\in \{0,\dots,T\}}|\gamma(s)|\le W\}}\right]\ge e^{(\lambda -\epsilon)(T-M(T))}\},$$
  and $F=\{|G_2|\ge \frac{|A|}{2}\}$.
 Due to condition (ii), $\lambda M(T)< \epsilon T$ for all $T$ sufficiently large. Therefore, 
 $$(\lambda -\epsilon)(T-M(T))\ge (\lambda -\epsilon)T-\epsilon T=(\lambda-2\epsilon)T.$$
 This gives
 \begin{equation}
 \label{eq:NMTMT}
   E_x \left[\exp(H(M(T),T))\ch_{\{\max_{s\in \{0,\dots,T\}}|\gamma(s)|\le W\}}\right]\ge e^{(\lambda -2\epsilon)T},\text{ for all }x\in G_2.
 \end{equation} 
By the spacing assumption on $A$ and the definition of $G_2$, the indicators $\{\ch_{G_2}(x)\}_{x\in A}$ form an IID sequence of Bernoulli trials. By Lemma \ref{lem:channels}, for every $x\in A$, $Q(\ch_{G_2}(x)=0)\le e^{-c(T-M(T))}$. However, by condition (ii), $T-M(T)\ge (1-\epsilon)T$, for all sufficiently large $T$. Therefore, letting $c'=(1-\epsilon)c/2$, we obtain $Q(\ch_{G_2}(x))\le e^{-2c' T}$. Next, note that 
\begin{align*}
 Q(F^c)&\le \sum_{A'\subset A,|A'|=\f{\frac {|A|}{2}}} Q(\sum_{x\in A'} \ch_{G_2}(x)=0)\le \binom{|A|}{\f{|A|/2}}e^{-2c'T \f{|A|/2}}\\
& \le  2^{|A|}e^{-c' T |A|/2}\le e^{-c' T |S_{M(T)}|/(8W)^d},~\text{ for all sufficiently large }T.
\end{align*}
By Lemma \ref{lem:ntx}, $|S_{M(T)}|\sim M(T)^d$, therefore there exists a constant $C_{\infty}>0$ depending only on $\epsilon,d$ and the distribution of $V(0,0)$ such that  $Q(F^c) \le e^{-C_{\infty}TM(T)^{1+d}}$ for all sufficiently large $T$. By definition, $F$ and $E$ are independent. Hence, 
$$ Q(E \cap F)\ge 1-e^{-\frac 12 \min(C_{\infty} T M(T)^d,J(T))},~\text{ for all sufficiently large }T.
$$
 On $F$, $|G_2|\ge \frac 12 |A| \ge \frac {|S_{M(T)}|}{2(8W)^d}$. Therefore, recalling that $r= 1-\frac{1}{4(8W)^d} $, 
$$|G_1 \cap G_2|= |G_1|+|G_2|-|G_1\cup G_2|\ge |S_{M(T)}| (1-\frac {1}{4(8W)^d})+\frac {1}{2(8W)^d}-1=\frac{|S_{M(T)}|}{4(8W)^d}>0.$$
   In particular, $|G_1\cap G_2|\ne \emptyset.$ 
Since
  $$Z(T)\ge E_0 e^{H(M(T))}\ch_{G_1\cap G_2}(\gamma(M(T)))E_x e^{H(M(T),T)},$$ 
 it follows that on $E \cap F$,  
\begin{align*}
 Z(T)&\ge 2^{-M(T)}e^{-\epsilon T} e^{(\lambda-2\epsilon)T}\\ 
 &\ge\exp ((\lambda -3\epsilon-\frac{M(T)}{T}\ln 2)T)\ge e^{(\lambda-4\epsilon)T}, \text{ for all sufficiently large }T.
\end{align*}
\end{proof} 
 We are ready to prove the proposition:
 \begin{proof}[Proof of Proposition \ref{pr:ubprop}]
 Let $C_1 =\max(2\eta_0,1)$. Let $\delta \in (0,1/2)$ be such that
   \begin{align}
   \label{eq:deltachosen1}
     &\delta C (1+C_0)\le \frac{\epsilon(1-r)}{2K}.\\
     \label{eq:deltachosen2} 
      &\max(\lambda,1)\frac{2\delta C}{C_1^{1/d}}< \epsilon
   \end{align}
   Let 
 $$\tilde \eta = \delta^d \eta\text{ and }M_{\delta}(T)= \c{\left(\frac{\delta \tilde \eta(T)}{C_1}\right)^{1/d}}.$$
  Clearly, for sufficiently large $T$,  $M_{\delta}(T)^d \le 2C_1^{-1}\delta \tilde \eta(T)$. Thus, 
    $$\frac{\tilde \eta(T)}{M_{\delta}(T)^d} \ge (2\delta)^{-1} C_1\ge 1.$$
  By Lemma \ref{lem:lll}, 
  \begin{align*}
  I_{\tilde \eta}^{M_{\delta}}(T)&\le F_{\tilde \eta}^{M_{\delta}}(T)+C_0 \tilde \eta(T)^{1/d}\\ 
 &\le \delta \eta(T)\int_{G^{inv}(1)}^{G^{inv}(\eta(T))}G^{-1/d}(x)dx + C_0\delta \eta^{1/d} (T) =\delta F(\eta(T))+C_0 \delta \eta^{1/d}(T)\\
 & \le \delta C (1+C_0)T,\text{ for all sufficiently large }T.
\end{align*}
     Hence by \eqref{eq:deltachosen1}
      $$\frac{I_{\tilde \eta}^{M_{\delta}}(T)}{T} \le \frac{\epsilon(1-r)}{2K}.$$
      The choice of $\tilde \eta$ and  $M_{\delta}$ satisfies the conditions of Lemma \ref{lem:markov}. Therefore,
 \begin{align*}
 \label{eq:markovgives}
 Q(E^c)&\le\exp(T \tilde \eta(T)( \frac {2^d \delta \ln 2}{T}+\frac{I_{\tilde \eta}^{M_{\delta}}}{T}-\frac{\epsilon (1-r)}{K}))\\
   & \le e^{-\frac{\epsilon (1-r)}{2K}\delta^d T \eta(T)},\text{ for all sufficiently large }T.
 \end{align*}
   Let $J(T)=\frac{\epsilon (1-r)}{2K}\delta^d T \eta(T)$. 
  By the definition of $M_{\delta}(T)$ and the fact that $\delta <1$, it follows that $M_{\delta}(T) \le \frac{2\delta \eta(T)^{1/d}}{C_1^{1/d}}$ for all sufficiently large $T$. Therefore by (i), $M_{\delta}(T)<  \frac{2\delta C T }{C_1^{1/d}}$. By \eqref{eq:deltachosen2} we have
    $$\frac{\max(\lambda,1) M_{\delta}(T)}{T}<\epsilon.$$
     Therefore by Lemma \ref{lem:glueing} 
      $$R_{\epsilon}(T) \ge \frac 12 \min(C_{\infty}T M_{\delta}(T)^d,J(T)).$$
       The claim follows because $M_{\delta}(T)^d\sim \eta(T)$ and $J(T)\sim T \eta(T)$. 
  \end{proof}
\subsection{Upper Bound.}
\label{sec:upperbound}
 By assumption (AS3), there exists some $q\in (0,1]$ such that 
 \begin{equation}
 \label{eq:lbV00new}
 Q(-V(0,0)\ge t)\ge q e^{-t G(t)},~\text{ for all t}\ge 0.
 \end{equation}
 This observation turns out to be very convenient.
The main result of this section is the following: 
 \begin{prop}
\label{pr:lb}
  and let $\epsilon,C \in (0,\infty)$ be constants and let $\eta:\Z_+ \to \R_+$ and  $M:\Z_+\to \Z_+$. Suppose that the following conditions hold: 
 \begin{enumerate}
\item $M(T)< \eta(T)^{1/d}<T$; 
\item $2\epsilon < \liminf_{T\to\infty} \frac{I_{\eta}^M(T)}{T} \le \limsup_{T\to\infty}  \frac{I_{\eta}^M(T)}{T} <C$;
\end{enumerate}
 Then 
 $$ \limsup_{T\to\infty} \frac{R_{\epsilon}(T)}{T \eta (T)}\le (\ln q^{-1}+C).$$
\end{prop} 
 Before proving the proposition, we obtain the following upper bounds: 
 \begin{lemma}~
 \label{lem:ubrate}
  \begin{enumerate}
 \item For every $\epsilon >0$, $\limsup_{T\to\infty} \frac{R_{\epsilon}(T)}{T^{1+d}}<\infty;$
 \item 
 $\limsup_{T\to\infty} \frac{R_{\epsilon}(T)}{TG(2\epsilon T)}<\infty$; 
\item  If $\epsilon \in (0,\frac 12)$ or if $f$ is non-increasing, then $\limsup_{T\to\infty} \frac{R_{\epsilon}(T)}{T G(T)}<\infty$.  
\end{enumerate}
\end{lemma}  
\begin{proof}
 We begin with (ii) and (iii).  Assume that $\epsilon>0$. 
 Let $A=\{-V(0,0)\ge 2\epsilon T\}$ and let 
 $$B=\cap_{|e|=1}\{E_e \exp H_{\gamma}(1,T)\ge e^{(\lambda+\epsilon)(T-1)}\}.$$ Hence, 
 \begin{equation}
  \label{eq:ZTsimple}
    Z(T) \le e^{-2\epsilon T+ (\lambda+\epsilon)(T-1)}\le e^{(\lambda-\epsilon)T},\text {on } A\cap B.
  \end{equation}
 It follows from Theorem \ref{th:subadditive}-(\ref{subadd2}) that $\lim_{T\to\infty} Q(B)=1$. Therefore, for all $T$ large enough, 
 $$Q(A \cap B)\ge \frac q2 e^{-2\epsilon T G(2\epsilon T)}.$$
 Thus (ii) follows from \eqref{eq:ZTsimple} and (iii) is an immediate consequence. 
\\
 To prove (i), we repeat the argument, redefining $A$ and $B$. Let $p=Q(-V(0,0)>4\epsilon)$. Note that by assumption $p\in (0,1)$. Let 
 $$A=\{-V(t,x)\ge  4\epsilon:t\in \{0, \c{T/2}-1\},x\in L_t\},$$ and   
 $$B=\cap_{x\in L_{\c{T/2}}}\{E_x e^{H( \c{T/2},T)}\le e^{(\lambda+\epsilon)(T-\c{T/2})}\}.$$
  Let $\gamma$ denote a random walk path with $\gamma(0)=0$. Then on the event $A$, 
 $$H_{\gamma}(\c{T/2})=\sum_{t=0}^{\c{T/2}-1} V(t,\gamma(t))\le -4\epsilon \c{T/2}\le -2\epsilon T.$$
  Thus, on $A\cap B$: 
 \begin{equation}
 \label{eq:ZTsimple2} 
Z(T)\le E_0 e^{H(\c{T/2})}\max_{x\in L_{\c{T/2}}}E_x e^{H(\c{T/2},T)}\le e^{-2\epsilon T}e^{(\lambda+\epsilon)(T-\c{T/2})}\le e^{(\lambda-\epsilon)T}.
 \end{equation}
We note that $B$ is the intersection of $|L_{\c{T/2}}|$ non-increasing events. Since $|L_{\c{T/2}}|<T^d$, it follows from the FKG inequality and Theorem \ref{th:subadditive}-(\ref{subadd2}) that 
  $$Q(B)\ge (1-e^{-c\c{T/2}})^{T^d}\ge 1-e^{-\frac c2 T}\underset{T\to\infty}{\to}1.$$
  We also have
   $$Q(A)\ge \prod_{t=0}^{ \c{T/2}-1} p^{|L_t|}\ge e^{-\ln p^{-1}\sum_{t=0}^{T}(1+t)^d}.$$
   Therefore there exists a constant $C_1>0$, depending only on $p$ and $d$ such that $Q(A)\ge e^{-C_1 T^{1+d}}$. Since the events $A$ and $B$ are independent, 
 $$ Q(A\cap B) \ge e^{-\frac{C_1}{2}T^{1+d}},\text{ for all sufficiently large }T.$$
 Thus, the claim follows from \eqref{eq:ZTsimple2}.
  \end{proof}
 We elaborate the argument in the above proof to obtain the following:

\begin{proof}[Proof of Proposition \ref{pr:lb}]
 Let
 $$A=\{-V(t,x)\ge G^{inv}(\frac{\eta(T)}{(1+t)^d}):x\in L_t,t\in \{0,\dots,M(T)-1\}\},$$
 and $$B=\cap_{x\in L_{M(T)}} \{E_x e^{H(M(T),T)}\le e^{(\lambda+\epsilon)(T-M(T))}\}.$$
 Due to condition (ii), on the event $A$
 $$H_{\gamma}(M(T))\le -\sum_{t=0}^{M(T)-1} G^{inv}(\frac{\eta(T)}{(1+t)^d})\le -2\epsilon T,$$
 for all paths $\gamma$ with $\gamma(0)=0$. In addition, $(\lambda+\epsilon)(T-M(T))\le (\lambda+\epsilon)T$. 
 Since 
$$Z(T)\le E_0 e^{H(M(T))}\max_{x\in L_{M(T)}}E_x e^{H(M(T),T)},$$
 it follows that 
\begin{equation}
 \label{eq:ZMTT} 
Z(T) \le e^{(\lambda -\epsilon)T}\text{ on }A\cap B.
\end{equation}
Next we estimate the probability of $A\cap B$ from below. First we observe that $B$ is an intersection of $|L_{M(T)}|$ identically distributed, non-increasing events. By (i), $|L_{M(T)}|\le (1+T)^d$. It follows from the FKG inequality and Theorem \ref{th:subadditive}-(\ref{subadd2}) that   
 \begin{equation}
 \label{eq:BMT}
 Q(B)\ge (1-e^{-c T})^{|L_{M(T)}|}\ge (1-e^{-cT})^{(1+T)^d}\underset{T\to\infty}{\to}1.
\end{equation}
 By \eqref{eq:lbV00new}, 
 \begin{align*}
 Q(A)&\ge \prod_{t=0}^{M(T)-1} \left(q \exp(-G^{inv}(\frac{\eta(T)}{(1+t)^d})G(G^{inv}(\frac{\eta(T)}{(1+t)^d})))\right)^{(1+t)^d}\\
 &\ge
 q^{M(T)^{1+d}} \exp(-\sum_{t=0}^{M(T)-1} G^{inv}(\frac{\eta(T)}{(1+t)^d})\frac{\eta(T)}{(1+t)^d}(1+t)^d)\\
 &=\exp(- M(T)^{1+d} \ln q^{-1}-\eta(T) I_{\eta}^M(T)).
\end{align*} 
 It follows from (i) that $M^{1+d}(T)\le \eta^{1/d+1}(T)\le T \eta(T)$. 
 By (ii),  $I_{\eta}^M(T)< C T$, for all sufficiently large $T$. Due to the independence of $A$ and $B$ and \eqref{eq:BMT},  
 $$Q(A\cap B)\ge \exp (-(\ln q^{-1}+C)T\eta(T)),\text{ for all sufficiently large T}.$$ 
 The claim follows from \eqref{eq:ZMTT}. 
\end{proof}
\subsection{Proof of  and Proposition \ref{pr:expo}, Theorem \ref{th:slow}, Theorem \ref{th:fast} and  Corollary \ref{cor:slowuseful}}
\begin{proof}[Proof of Proposition \ref{pr:expo}]
 Fix $\epsilon >0$ and let $W$ be as in Lemma \ref{lem:channels}. Clearly,  $Z(T)\ge  E_0 [\exp(H(T))\ch_{\{\max_{s\in \{0,\dots,T\}}|\gamma(s)|\le W\}}]$. Therefore $\liminf_{T\to\infty} \frac{R_{\epsilon}(T)}{T}>0$. 
 On the other hand, the argument in  Lemma \ref{lem:ubrate}-(ii) applies here as well, which shows that $\limsup_{T\to\infty} \frac{R_{\epsilon}(T)}{T}<\infty$.
\end{proof}
\begin{proof}[Proof of Theorem \ref{th:slow}]~\\
\noindent(i). 
 Suppose that $\gamma <\infty$. By Lemma \ref{lem:ubrate}-(iii), 
  $$\limsup_{T\to\infty} \frac{R_{\epsilon}(T)}{TG(T)}<\infty.$$
  Let $\eta(T)=G(T)$. We apply Proposition \ref{pr:ubprop}. Condition (ii) is satisfied because $\gamma<\infty$. Condition (i) is satisfied because $f$ is non-increasing.  Thus, the proposition gives
   $$\liminf_{T\to\infty} \frac{R_{\epsilon}(T)}{T G(T)}>0.$$
\noindent(ii). Suppose that $\gamma=\infty$. Let $\delta \in (0,\epsilon)$ and set $\eta(T)=G(\delta T)$. We wish to find a function 
 $M:\Z_+\to \Z_+$ such that $I_{\eta}^M(T)\in [2\epsilon,3\epsilon)$. 
 On the one hand, if $M\equiv 1$, then $I_{\eta}^M(T)=\delta T <\epsilon T$.
 On the other hand, if $M(T)=\f{\eta^{1/d}(T)}$,  then by Lemma \ref{lem:lll}, $I_{\eta}^M(T)\ge F(G(\delta T))-\delta T$, which shows that $I_{\eta}^M(T)/T \to \infty$ as $T\to\infty$. 
  For $m\in \N$, let  $I^{m}(T)=I_{\eta}^M$, where $M\equiv m$. 
 $$I^{m+1}(T)-I^{m}(T)= G^{inv}(\frac{G(\delta T)}{(1+m)^d})\le G^{inv}(G(\delta T))=\delta T<\epsilon T.$$
 It follows that for every $T$ sufficiently large, there exists a choice of $M$ such that $1\le M(T)<\f{G(\delta T)^{1/d}}$ and that $I_{\eta}^{M}\in [2\epsilon T,3\epsilon T)$. Therefore both conditions of Proposition \ref{pr:lb} are satisfied with $C=3\epsilon$ and it follows from the proposition that  
  $$\limsup_{T\to\infty} \frac{R_{\epsilon}(T)}{T G(\delta T)} \le \ln q^{-1}+3\epsilon.$$
   \end{proof}
\begin{proof}[Proof of Theorem \ref{th:fast}]
 \item Since $F$ is strictly increasing, continuous and has $F(G^{inv}(1))=0$, $\lim_{z\to\infty} F(z)=\infty$,  there exists $\eta_1:\Z_+\to \R_+$ such that $F(\eta_1(T))=3\epsilon T$. By definition,  
     $$\frac{\eta_1(T)^{1/d}(T)}{2\epsilon T}=\left(\int_{G^{inv}(1)}^{\eta_1(T)} G^{-1/d}(x) dx\right)^{-1}.$$
    Therefore, $\limsup_{T\to\infty} \frac{\eta_1(T)^{1/d}}{T}<\infty$.
    By Proposition \ref{pr:ubprop}
      $$\liminf_{T\to\infty} \frac{R_{\epsilon}(T)}{T \eta_1(T)}>0.$$
 Let $M(T)=\c{\eta(T)^{1/d}}$. It follows from Lemma \ref{lem:lll} that $I_{\eta_1}(T) \ge F(\eta_1(T))+\Delta(T)$. Since $f$ is non-decreasing, $\Delta(T)\ge 0$. Therefore, $I_{\eta_1}(T) \ge 3\epsilon T$. In addition, Lemma \ref{lem:lll} shows that $I_{\eta_1}(T) \le F_{\eta_1}^M(T)+C_0 \eta_1^{1/d}(T)$. Note that
 $$ F_{\eta_1}^M(T)=F(\eta_1(T))+\eta_1(T)^{1/d}\int_{G^{inv}(\eta(T)/M(T)^d)}^{G^{inv}(1)}G^{-1/d}(x)dx=F(\eta_1(T))+o(\eta(T)^{1/d}).$$
 Therefore,
  $$I_{\eta_1}^M \le F(\eta_1(T))+(1+C_0)\eta_1(T)^{1/d},$$
  for all $T$ sufficiently large. Thus, the conditions of  Proposition \ref{pr:lb} are satisfied and we have
        $$\limsup_{T\to\infty} \frac{R_{\epsilon}(T)}{T \eta_1(T)}<\infty.$$
        It is left to show that $\eta\sim \eta_1$. Equivalently, we need to show that $\limsup_{T\to\infty} \frac{\eta_1(T)}{\eta(T)}<\infty$ and $\limsup_{T\to\infty} \frac{\eta(T)}{\eta_1(T)}<\infty$.  We will only prove the first inequality, the argument being identical. We argue by contradiction. If there exists a sequence $t_k\nearrow \infty$ such that $\eta_1(t_k)\ge k^d\eta(t_k)$, then $ \frac{F(\eta_1(t_k))}{F(\eta(t_k))}\ge k\underset{k\to\infty}{\to}\infty$, contradicting the fact that  $F(\eta_1(T))\sim T \sim  F(\eta(T))$.
\end{proof}
\begin{proof}[Proof of Corollary \ref{cor:slowuseful}]
 Let $u(y)=f^{-1/d}(y),v(y)=\int_{G^{inv}(1)}^y G^{-1/d}(x)dx$.
  Note that 
 $$ \frac{v'(y)}{u'(y)}=\frac{G^{-1/d}(y)}{-\frac 1d f^{-(1+1/d)}(y)f'(y)}= -d \frac{f(y)}{y f'(y)}.$$
 Therefore $\lim_{y\to\infty} \frac{u'(y)}{v'(y)}=\frac{d}{\rho}$.  
 Since $\lim_{y\to\infty} u(y)=\infty$, $\lim_{y\to\infty} v(y)=\infty$, it follows from L'Hospital's rule that  $\gamma = \lim_{y\to\infty} \frac{v'(y)}{u'(y)}=d/\rho$. Therefore the first claim follows from Theorem \ref{th:slow}-(\ref{slow1}). To prove the second claim, assume that $\rho=0$. Since $f$ is convex, for all $y,\delta >0$,  $f(y)\ge f(\delta y)+(1-\delta)y f'(\delta y)$. Therefore, 
$$ \frac{G(y)}{G(\delta y)}=\frac{f(y)}{\delta^{d}f(\delta y)}\ge \delta^{-d} \left(1+\frac{1-\delta}{\delta}\frac{\delta y f'(\delta y)}{f(\delta y)}\right)\underset{y\to\infty}{\to}\delta^{-d}.$$
 This implies that $\limsup_{T\to\infty} \frac{G(\delta T)}{G(T)}\le \delta^{d}$. Therefore by Theorem \ref{th:slow}-(\ref{slow2}), 
 $$\limsup_{T\to\infty} \frac {R_{\epsilon}(T)}{TG(T)}=0.$$
\end{proof} 
\section*{Appendix}
All proofs in this section are carried out in one dimension, the extension to higher dimensions being immediate. \\
   For non-negative integers $L$ and $t_1\le t_2$, and for $x\in \Z^d$ we let 
\begin{equation*}
   C_{t_1,t_2,L}(x)=\{\gamma:\gamma(t_1)=\gamma(t_2)=x,~\max_{s\in\{t_1,\dots,t_2\}}|\gamma(s)-x|\le L\},
   \end{equation*}
   and 
 $$B_{t_1,t_2,L}(x)= \{t_1,\dots,t_2\}\times \{z\in \Z^d:|z-x| \le L\}.$$
  We say that $B_{t_1,t_2,L}(x)$  is $\epsilon$-good if 
   $$ E_x [\exp(H(t_1,t_2))\ch_{C_{t_1,t_2,L}(x)}]\ge e^{(\lambda -\epsilon)(t_2-t_1)},$$ 
 To prove Lemma \ref{lem:channels}, we build on the following: 
 \begin{lemma} \label{lem:goodblock}
 Assume that (AS0)-(AS2) hold.  For every $\epsilon,\delta>0$, one can choose  $L=L(\epsilon,\delta),~W=W(\epsilon,\delta)$ such that $Q(B_{0,L,W}(0)\text{ is }\epsilon-\text{good})>1-\delta$. In addition, for every fixed $\epsilon,\delta$ and a corresponding value of $W$, the ratio $L/W$ can be made arbitrarily large.   
\end{lemma} 
 \begin{proof}
 Let $U_1$ and $U_2$ be two identically distributed monotone functions of $V$ and let $K>0$ be a constant. Suppose that 
  $$Q(U_1+U_2 \ge K)\ge 1-\delta',$$
   for some $\delta'>0$. 
    Then $$Q(U_1+U_2 \ge K)\le Q(U_1 \ge K/2\text{ or } U_2 \ge K/2)=2Q(U_1 \ge K/2)-Q(U_1\ge K/2,U_2\ge K/2).$$
      By the FKG inequality, we obtain  
  $$1-\delta' \le Q(U_1+U_2 \ge K)\le Q(U_1 \ge K/2)(2-Q(U_1\ge K/2)),$$ 
 from which it follows that 
  $$Q(U_1 \ge K/2) \ge 1-\sqrt{\delta'}.$$
  By Theorem \ref{th:subadditive}-(\ref{subadd1}),    
  $$Q(Z(T) > e^{(\lambda-\epsilon) T})>1-\delta',\text{ for sufficiently large }T.$$  
 Let $U_1=E_0 \exp(H(T))\ch_{\{\gamma(T)\ge 0\}}$, $U_2=E_0 \exp(H(T))\ch_{\{\gamma(T)\le 0\}}$. Clearly, $U_1$ and $U_2$ are identically distributed monotone functions of $V$ and $Q(U_1+U_2\ge e^{(\lambda-\epsilon)T})\ge 1-\delta'$. 
Thus, 
 \begin{equation} 
   \label{eq:cornerstone} 
   Q(E_0  \exp(H(T))\ch_{\{\gamma(T)\ge 0\}}> e^{(\lambda-2\epsilon)T})>1-\sqrt{\delta'} \text{ for sufficiently large }T.
\end{equation} 
  Below, we denote by ${\cal G}_n$ the $\sigma$-algebra generated by $\{V(t,x): (t,x)\in \{0,\dots,n\}\times \Z\}$. 
 Set $x_0^*=0$ and let  $x_1^*$ be a measurable function of ${\cal G}_{T-1}$ with the property 
 $$E_0 \exp(H(T)) \ch_{\{\gamma(T)=x_1^*\}}=\max_{x\ge 0}E_0 \exp (H(T)) \ch_{\{\gamma(T)=x\}}.$$ 
 Now 
   $$E_0 \exp(H(T)) \ch_{\{\gamma(T)=x_1^*\}}\ge  
   \frac {1}{T+1} E_0  \exp(H(T))\ch_{\{\gamma(T)\ge 0\}}.$$ 
Thus, it follows from \eqref{eq:cornerstone} that  
 $$Q(E_0 \exp(H(T))\ch_{\{\gamma(T)=x_1^*\}}>e^{(\lambda-3\epsilon)T})>1-\sqrt{\delta'}\text{ for sufficiently large }T.$$
 We define the function $\sign:\Z \to \{-1,1\}$ be letting $\sign(z)=1$ if and only if $z>0$.    We continue the construction by induction. Having defined $x_k^*$, we let  $x_{k+1}^*$ be a measurable function of ${\cal G}_{(k+1)T-1}$ with the properties 
\begin{enumerate}
\item \label{cond:navigate} 
 If $x_k^*\ge 0$, then $x_{k+1}^*\le x_k^*$. Otherwise, $x_{k+1}^*\ge x_k^*$.  
 \item 
 \label{cond:maximize}  
\begin{align*} 
 E_{x_k^*}&\exp(H(kT,(k+1)T))\ch_{\{\gamma(T)=x_{k+1}^*\}}\\ 
= 
      &\max_{\{x\in \Z: (x_k^*-x)\sign(x_k^*)\ge 0\}}E_{x_k^*} \exp(H(kT,(k+1)T))\ch_{\{\gamma(T)=x\}}. 
 \end{align*} 
 \end{enumerate} 
 Note that condition (\ref{cond:navigate}) guarantees that $|x_k^*|\le T$ for all $k$. 
 Our construction also satisfies that on $\{x_k^*=l\}$, $E_{x_k^*} \exp (H(kT,(k+1)T))\ch_{\{\gamma(T)=x_{k+1}^*\}}$ has the same distribution as $E_0 \exp (H(T))\ch_{\{\gamma(T)=x_1^*\}}$. In particular, 
 \begin{align*} 
 Q&(E_{x_k^*}\exp (H(kT,(k+1)T))\ch_{\{\gamma(T)=x_{k+1}^*\}}\ge e^{(\lambda -3\epsilon)T})\\ 
 &=\sum_{l}Q(\{x_k^*=l\}\cap \{E_l \exp (H(kT,(k+1)T))\ch_{\{\gamma(T)=x_{k+1}^*\}}\ge e^{(\lambda -3\epsilon)T}\}). 
 \end{align*} 
 However, the event $\{E_l \exp (H(kT,(k+1)T))\ch_{\{\gamma(T)=x_{k+1}^*\}}\ge e^{(\lambda -2\epsilon)T}\}$ depends only on $\{V(t,x):t\ge kT,x\in \Z\}$, whereas $x_k^*\in {\cal G}_{kT-1}$. Therefore we conclude that  
 $$ Q(E_{x_k^*}\exp (H(kT,(k+1)T))\ch_{\{\gamma(T)=x_{k+1}^*\}}\ge e^{(\lambda -3\epsilon)T})\ge 1-\sqrt{\delta'}.$$ 
 For $R\in\N$,  let $Z_R = E_0 \exp (H(RT))\prod_{k=1}^R \ch_{\{\gamma(kT)=x_k^*\}}$. By the Markov property,  
 $$Z_R = E_0 \prod_{k=0}^{R-1} E_{x_k^*} \exp(H(kT,(k+1)T))\ch_{\{\gamma(T)=x_{k+1}^*\}}.$$ 
 Since $Q\left(\cup_{k=0}^{R-1}\left\{E_{x_k^*} \exp(H(kT,(k+1)T))\ch_{\{\gamma(T)=x_{k+1}^*\}}<e^{(\lambda -3\epsilon)T}\right\}\right)\le R \sqrt{\delta'}$, it follows that 
  \begin{equation} 
   \label{eq:ZR} 
    Q(Z_R>e^{(\lambda-3\epsilon)RT})\ge 1-R \sqrt{\delta'}. 
     \end{equation} 
 We also observe that due to the fact that $|x_k^*|\le T$, all paths $\gamma$ considered in the expectation defining $Z_R$ satisfy $\max_{j\in \{0,\dots,RT\}}|\gamma(j)|\le 2T$. Therefore,   
 $$\ch_{C_{0,(R+1)T,2T}(0)}\ge \ch_{\{\gamma(0)=0\}}\ch_{\{\gamma((R+1)T)=0\}}\prod_{k=1}^R \ch_{\{\gamma(kT)=x_k^*\}}.$$ 
 This implies that 
 \begin{equation} 
  \label{eq:block_lowbound} 
   E_0\exp (H((R+1)T))\ch_{C_{0,(R+1)T,2T}(0)} \ge Z_R \min_{z\in L_{RT},~|z|\le T} W_z, 
   \end{equation} 
  where $W_z =E_z \exp (H(RT,(R+1)T))\ch_{\{\gamma(T)=0\}}.$ 
  We assume from now that $T$ is even.  For $z\in L_{RT}$ with $|z|\le T$,  let $\gamma_z$ denote an arbitrary path with $\gamma_z(0)=z,~\gamma_z(T)=0$. Since $T$ is even, there exists such a path.  Clearly, $W_z \ge 2^{-T} e^{H_{\gamma_z}(RT,(R+1)T)}$. Hence, 
 $$\min_{z\in L_{RT},|z|\le T} W_z \ge 2^{-T} \exp (\min_{z\in L_{RT},|z|\le T} H_{\gamma_z}(RT,(1+R)T)).$$
  Since $H_{\gamma_z}(RT,(1+R)T)$ and $\sum_{k=0}^{T-1} V(k,0)$ are identically distributed, we have 
   $$Q(\min_{z\in L_{RT},|z|\le T} H_{\gamma_z}(RT,(1+R)T) \le -\epsilon T) \le (1+T) Q(-\sum_{k=0}^{T-1} V(k,0) \ge \epsilon T).$$
  Since $Q(V(0,0))=0$, it follows that for all $|\mu|$ small enough, $Q(e^{-\mu V(0,0)})\le e^{c\mu^2}$, for some $c\le Q(V(0,0)^2)$. Hence,
\begin{equation}
 \label{eq:minildp}
 Q(-\sum_{k=0}^{T-1} V(k,0) \ge \epsilon T)\le e^{c T\mu^2}e^{-\mu\epsilon T} \le e^{-c'T},
\end{equation}
for some $c'>0$, depending only on $\epsilon,c$ and $\mu$. Consequently, 
 $$Q(\min_{z\in L_{RT},|z|\le T} H_{\gamma_z}(RT,(1+R)T) \le -\epsilon T)\le e^{-cT},\text{ for all sufficiently large }T.$$
   It follows from \eqref{eq:ZR} and \eqref{eq:block_lowbound} that  
     $$Q(E_0\exp (H((R+1)T))\ch_{C_{0,(R+1)T,2T}(0)}\ge e^{(\lambda -3\epsilon)RT-(\epsilon+\ln 2) T})\ge 1-R\sqrt{\delta'}-e^{-c T}.$$
 The first statement of the lemma follows by adjusting $R$ and $\delta'$ appropriately and setting  $L=(T+1)R$ and $W=2T$. The second statement follows from the fact that for every $T$, $R$ can be arbitrarily large.   
\end{proof} 
\begin{proof}[Proof of Lemma \ref{lem:channels}]
  By Lemma \ref{lem:goodblock} we may choose $\delta$ sufficiently small and $W$ and $L$ sufficiently large such that  $Q(B_{0,L,W}(0)\text{ is }\epsilon\text{-good})\ge 1-\delta$.  We will choose $\delta$ and $L$ as function of $\epsilon$ which will be determined later, taking values in the even positive integers.  At the moment we only require $\eta\equiv \epsilon -2\delta \ln 2$ be strictly positive.
 Let $X_k = E_0 \exp(H(kL,(k+1)L))\ch_{C_{kL,(k+1)L,W}(0)}$. Let 
  $$A=\{\exists B \subset \{0,\dots, n-1\},~|B|\le 2\delta n,~\prod_{k\in B} X_k \le e^{-\epsilon nL}\}.$$ 
 Let $\gamma$ be any path with $\gamma(kL)=0$ for all $k$. 
 Then $X_k \ge 2^{-L} \exp (\sum_{j=kL}^{(k+1)L-1} V(j,\gamma(j)))$.
 Therefore for every $B$, 
 $$\{\prod_{k\in B} X_k < e^{-\epsilon n L}\}\subset 
   \{\sum_{k\in B}\left( -L\ln 2 + \sum_{j=kL}^{(k+1)L-1}V(j,\gamma(j))\right)\le -\epsilon n L\}.$$
 We obtain  
 $$Q(\{\prod_{k\in B} X_k < e^{-\epsilon n L}\})\le Q(E_{|B|}),\text{ where }E
_{|B|}=
  \{-\sum_{j=0}^{L|B|-1}V(j,0)\ge  \eta n L \}.$$
  By the Markov inequality, for every $\mu>0$,  
 $$Q(E_{|B|})\le Q(e^{-\mu V(0,0)})^{L|B|}e^{-\mu \eta n L}\le Q(e^{-\mu V(0,0)})^{2\delta n L}e^{-\mu \eta nL},$$
 where in the last inequality we have used the fact that $Q(e^{-\mu V(0,0)})\ge 1$ that $|B|\le 2\delta n$.  By \eqref{eq:minildp}, it follows that 
 by choosing $\mu$ sufficiently small, there exists a constant $c_1>0$, depending only on $\eta$ and $\delta$ such that $Q(E_{|B|}) \le e^{-c_1 n L}$. 
 Consequently, 
 $$Q(A)\le \sum_{B\subset \{0,\dots,n-1\},~|B|\le 2\delta n}Q(E_{|B|}) \le
 \sum_{k=1}^{\f{2\delta n}} \binom{n}{k} e^{-c_1 nL}\le 2^{n}e^{-c_1 nL}.$$
 By letting $L\ge \frac {2}{c_1} \ln 2$, we obtain 
 $$Q(A)\le e^{-n \ln 2}.$$
 Let 
 $$C=\{\exists G\subset \{0,\dots,n-1\}:|G|\ge (1-2\delta)n,\text{ such that for all }k\in G,~ X_k\ge e^{(\lambda-\epsilon)L}\}.$$
 The event $C$ is the event that the number of successes in $n$ IID Bernoulli trials is at least $(1-2\delta)n$, where a success in the $k$'th trial is the event $\{X_k \ge e^{(\lambda-\epsilon)L}\}$. By definition of the $X_k$'s, the probability of success is bounded below by $1-\delta$. 
 Therefore, there exists a constant $c_2>0$, depending only on $\delta$ such that 
 $$Q(C) \ge 1-e^{-c_2 n}.$$
 Since $A^c$ and $C$ are non-decreasing events, it follows that $Q(A^c\cap C) \ge 1-e^{-\frac 12 \min(c_1,c_2)n}$. We now require that $(\lambda-\epsilon)(1-2\delta)\ge \lambda-2\epsilon$. This can be achieved by choosing $\delta$ sufficiently small. With such a choice, on  $A^c\cap C$ 
 $$ \prod_{k=0}^{n-1} X_k \ge e^{(\lambda-\epsilon)(1-2\delta)nL}e^{-\epsilon n L}\ge e^{(\lambda-3\epsilon)nL}.$$
 Finally, 
 $$ Z(nL)= E_0 e^{H(nL)}\ge \prod_{k=0}^{n-1} X_k,$$
 completing the proof for  $T$ of the form $nL$.  The extension to all large $T$ is simple and will be omitted.  
\end{proof}
\subsection*{Acknowledgment} I would like to thank Mike Cranston, Demian Gauthier and Thomas Mounford for showing me \cite{CGM} prior to its publication. 
\bibliographystyle{amsalpha}
\bibliography{ldp_pf}

\providecommand{\bysame}{\leavevmode\hbox to3em{\hrulefill}\thinspace}
\begin{thebibliography}{CMS02}

\bibitem[AD95]{aldous}
D.~Aldous and P.~Diaconis, \emph{Hammersley's interacting particle process and
  longest increasing subsequences}, Probab. Theory Related Fields \textbf{103}
  (1995), no.~2, 199--213.

\bibitem[CGM]{CGM}
M.~Cranston, D.~Gauthier, and T.S. Mountford, \emph{On large deviations regimes
  for random media models}, preprint.

\bibitem[CH02]{CH}
Philippe Carmona and Yueyun Hu, \emph{On the partition function of a directed
  polymer in a {G}aussian random environment}, Probab. Theory Related Fields
  \textbf{124} (2002), no.~3, 431--457.

\bibitem[CMS02]{CMS}
M.~Cranston, T.~S. Mountford, and T.~Shiga, \emph{Lyapunov exponents for the
  parabolic {A}nderson model}, Acta Math. Univ. Comenian. (N.S.) \textbf{71}
  (2002), no.~2, 163--188.

\bibitem[CSY03]{CSY}
Francis Comets, Tokuzo Shiga, and Nobuo Yoshida, \emph{Directed polymers in a
  random environment: path localization and strong disorder}, Bernoulli
  \textbf{9} (2003), no.~4, 705--723.

\bibitem[CZ03]{CZ}
Yunshyong Chow and Yu~Zhang, \emph{Large deviations in first-passage
  percolation}, Ann. Appl. Probab. \textbf{13} (2003), no.~4, 1601--1614.

\bibitem[DZ99]{zeitouni}
Jean-Dominique Deuschel and Ofer Zeitouni, \emph{On increasing subsequences of
  {I}.{I}.{D}.\ samples}, Combin. Probab. Comput. \textbf{8} (1999), no.~3,
  247--263.

\bibitem[Kes85]{kesten}
Harry Kesten, \emph{First-passage percolation and a higher-dimensional
  generalization}, Particle systems, random media and large deviations
  (Brunswick, Maine, 1984), Contemp. Math., vol.~41, Amer. Math. Soc.,
  Providence, RI, 1985, pp.~235--251.

\end{thebibliography}
\end{document}